\documentclass[reqno, english]{amsart}
\usepackage{etex}
\usepackage{amsmath,amssymb,amsthm,bbm,mathtools,comment}

\usepackage[shortlabels]{enumitem}
\usepackage[pdftex,colorlinks,backref=page,citecolor=blue]{hyperref}
\hypersetup{pdfpagemode=UseNone,pdfstartview={XYZ null null 1.00}}
\usepackage[mathscr]{euscript}
\usepackage[usenames,dvipsnames]{color}
\usepackage{adjustbox,tikz,calc,graphics,babel,standalone}
\usetikzlibrary{shapes.misc,calc,intersections,patterns,decorations.pathreplacing}
\usetikzlibrary{arrows,shapes,positioning}
\usetikzlibrary{decorations.markings}
\usepackage{tikz}
\usepackage[final]{microtype}
\usepackage[numbers]{natbib}
\usepackage{cmtiup}
\usepackage{amsfonts}
\usepackage{graphicx}
\usepackage{caption}
\usepackage{subcaption}
\usepackage{verbatim}
\usepackage{array}
\usepackage[frame,cmtip,arrow,matrix,line,graph,curve]{xy}
\usepackage{graphpap, color, pstricks}
\usepackage{pifont}
\usepackage[final]{microtype}
\usepackage{cmtiup}
\usepackage{todonotes}
\usetikzlibrary{topaths,calc} 
\tikzstyle{vertex} = [fill,shape=circle,node distance=80pt]
\tikzstyle{edge} = [opacity=0.4,fill opacity=0.0,line cap=round, line join=round, line width=40pt]
\tikzstyle{elabel} =  [fill,shape=circle,node distance=30pt]

\allowdisplaybreaks

\theoremstyle{plain}
\newtheorem{theorem}{Theorem}[section]		
\newtheorem{lemma}[theorem]{Lemma}
\newtheorem{claim}[theorem]{Claim}
\newtheorem{proposition}[theorem]{Proposition}
\newtheorem{corollary}[theorem]{Corollary}
\newtheorem{conjecture}[theorem]{Conjecture}
\newtheorem{problem}[theorem]{Problem}

\newtheorem{definition}[theorem]{Definition}
\theoremstyle{remark}
\newtheorem*{remark}{Remark}

\def\TTT{\mathcal{T}}
\def\RR{\mathcal{R}}

\def\HH{\mathcal{H}}
\def\GG{\mathcal{G}}
\def\E{\mathbb{E}}
\def\bJ{\mathbf{J}}

\def\bB{\mathbf{B}}
\def\bD{\mathbf{D}}

\newcommand{\eps}{\ensuremath{\varepsilon}}

\newcommand{\Prob}{\mathbb{P}}
\newcommand{\dist}{\text{dist}}

\newcommand{\degi}{\deg^{\text{in}}}
\newcommand{\dego}{\deg^{\text{out}}}
\newcommand{\No}{N^{\text{out}}}
\newcommand{\Ni}{N^{\text{in}}}
\newcommand{\Vin}{V_{\text{in}}}
\newcommand{\Vout}{V_{\text{out}}}
\newcommand{\Nm}{N^{-}}

\newcommand{\rind}{r_{\text{ind}}}
\newcommand{\rsi}{\hat{r}}
\newcommand{\rsind}{\hat{r}_{\text{ind}}}

\let\originalleft\left
\let\originalright\right
\renewcommand{\left}{\mathopen{}\mathclose\bgroup\originalleft}
\renewcommand{\right}{\aftergroup\egroup\originalright}

\makeatletter
\def\imod#1{\allowbreak\mkern10mu({\operator@font mod}\,\,#1)}
\makeatother

\title{Embedding induced trees in sparse expanding graphs}
\author{Ant\'onio Gir\~ao}
\author{Eoin Hurley}

\thanks{
AG: Mathematical Institute, University of Oxford, Oxford OX2 6GG, UK. E-mail: {\tt girao@maths.ox.ac.uk}. {Research supported by EPSRC grant EP/V007327/1} and ERC Advanced Grant no. 883810}
\thanks{
EH: Korteweg de Vries Instituut, Universiteit van Amsterdam,  Science Park 904, 1098 XH Amsterdam, The Netherlands. E-mail: \textsf{eoin.hurley@umail.ucc.ie}} 

\begin{document}
\maketitle
\begin{abstract}
Inspired by the network routing literature \cite{aggarwal1996efficient}, we develop what we call a ``Pre-Emptive Greedy Algorithm" to embed bounded degree induced trees in sparse expanders. 
 This generalises a powerful and central result of Friedman and Pippenger to the induced setting.
 As corollaries we obtain that a sparse random graph contains all bounded  degree trees of linear order (whp) and that the induced and size induced Ramsey numbers of bounded degree trees are linear. 
No such linear bounds were previously known. 
We also prove a nearly-tight result on induced forests in bounded degree countable expanders.
We expect that our new result will find many more applications. 
\end{abstract}

% \red
% To do:
% \begin{itemize}
%     \item Add accents to names
%     \item Add statements of main results
%     \item Remove girth and max degree condition if possible.
%     \item add concentration theorem and LLL statement
%     \item Add discusssion of universality
%     \item Proof read
%     \item lower bound for induced ramsey numbers of bounded degree graphs?
%     \item Add open questions: ramsey Treewidth and improved bounds for G(n,d/n)? Hypergraph power of trees \cite{letzter2021size}.
% \end{itemize}

% \black

\section{Background}
Over the past century, a central theme in Combinatorics has been to find the right conditions that guarantee the existence of specific subgraphs.
Examples include Dirac's Theorem \cite{dirac1952some} from the $40$'s, Ramsey theory \cite{ramsey1928problem}, universal graphs \cite{moon1965minimal} and multiple questions in random graph theory going back as far as the birth of the subject itself \cite{erdHos1960evolution}.
On the applications side, the travelling salesman problem or network routing problems can both be phrased in terms of finding or containing appropriate subgraphs.

Dirac's Theorem is perhaps \textit{the} classical result in this area, it states that \emph{every} $n$-vertex graph with minimum degree at least $n/2$ contains a Hamilton cycle. Thirty five years later P\'osa \cite{posa1976hamiltonian} gave the first proof of Hamiltonicity (with high probability) in sparse random graphs of average degree $\Omega(\log n)$. 
Implicit in his work was the following fundamental  deterministic result ($N(X)$ is the set of vertices with a neighbour in $X$).
\begin{theorem}[P\'osa]
    Let $G$ be a graph such that every subset $X\subset V(G)$ with at most $n$ vertices satisfies $|N(X)| > 3|X| - 1$, then $G$ contains a path/cycle on at least $3n-2$ vertices.
\end{theorem}
This result (and the method of P\'osa rotations) quickly found other applications, most notably in Beck's proof \cite{beck1983size}  that the size ramsey number of paths is linear in its order. 
This answered a question from the seminal paper of Erd\H{o}s, Faudree, Schelp and Rousseau \cite{erdHos1978size} for which Erd\H{o}s had later offered $\$100$ \cite{erdHos1981combinatorial}. 
Beck conjectured that such a linear upper bound should hold not just for paths but for any bounded degree tree (see Section \ref{sec:induced_ramsey}), but could not prove it. 
Finally, Friedman and Pippenger \cite{friedman1987expanding} confirmed this by showing the following beautiful generalisation of P\'osa's Theorem.

\begin{theorem}
    Let $G$ be a graph such that  every set $X$ with at most $2n-2$ vertices satisfies $|N(X)| \geq (d+2)|X|$, then $G$ contains every tree with maximum degree at most $d$ on $n$ vertices.  
\end{theorem}
This proved to be a fundamental result with many applications in graph theory and in the network routing theory \cite{feldman1988wide}. 
It was used by Alon, Krivelevich and Sudakov \cite{alon2007embedding} to embed almost spanning trees in sparse random graphs. 
The method was later refined by Haxell \cite{haxell2001tree} and modified by Glebov, Johannsen and Krivelevich \cite{glebov8hitting} before it was used by Montgomery \cite{montgomery2019spanning} to show that $G(n,C_\Delta\log n/n)$ 
 contains every spanning tree of maximum degree $\Delta$ with high probability, thus resolving a conjecture of Kahn. 
Friedman and Pippinger's method has also been recently used by Dragani\'c, Krivelevich and Nenadov~\cite{draganic2022rolling} in conjunction with a rolling back technique to embed a variety of very sparse graphs in expanders.
 It was used even more recently, (again with rolling back) by Dragani\'c, Montgomery,  Munh\'a Correia, Pokrovskiy and Sudakov \cite{DMCPS2024} to prove a long standing conjecture of Krivelevich and Sudakov stating that pseudorandom graphs are Hamiltonian.
Finally,  it is a key ingredient in many bounds in size ramsey theory for bounded degree (hyper)trees \cite{haxell1995size}, (hyper)graphs of bounded treewidth \cite{berger2021size,kamcev2021size,HunterSudakov}, and other related families of sparse (hyper)graphs \cite{letzter2021size}.

Finding \textit{induced} subgraphs, while equally natural, is more challenging than the non-induced case and our understanding typically lags behind. 
In the case of bounded degree trees and related sparse graphs, a technical but significant reason for this gap in our understanding is the lack of any induced analogue of Friedman and Pippenger's powerful result. This is our main contribution.

Note that both in P\'osa's and Friedman and Pippenger's results there is only a lower bound on order of the vertex boundary of subsets. Thus large cliques satisfy the conditions of both theorems although they do not contain any non-trivial induced tree. 
We must therefore add an upper bound on average degree of small subgraphs, that is, we must forbid small dense spots.
We also need a maximum degree condition that is trivial to satisfy in all of our applications but we believe it to be a mere artifact of our proof (since we use the Lov\'asz Local Lemma).
One final difference with P\'osa's and Friedman and Pippenger's results is an extra factor of $\Delta$ which we conjecture could be removed completely, but removing it will likely require some new ideas.
\begin{theorem}\label{thm:main_simp}
    Let $G$ be a graph with minimum degree $10^7\Delta$ and maximum degree at most $\exp(\Delta/10^9)$ such that every subgraph on at most $(10^7\Delta+1) n$ vertices has average degree at most $12/5$.
    Then $G$ contains every tree with maximum degree at most $\Delta$ on $n$ vertices as an induced subgraph. 
\end{theorem}
 
One drawback of Friedman and Pippenger's theorem is  that it is non-algorithmic.
This was rectified by Dellamonica and Kohayakawa \cite{dellamonica2008algorithmic} who (under a slightly stronger and more robust expansion condition) reduced the problem to that of finding a matching in a robustly expanding bipartite graph. 
They then applied an algorithmic (also online) result of
Aggarwal, Bar-Noy, Coppersmith, Ramaswami, Schieber and Sudan \cite{aggarwal1996efficient} on such matchings (the motivation of \cite{aggarwal1996efficient} was to obtain algorithmic versions of  Friedman and Pippenger's applications in the network routing literature).
It is from this latter proof \cite{aggarwal1996efficient} that we take our inspiration and we call the overall strategy the ``Pre-Emptive Greedy Algorithm". 
Our result yields an efficient online algorithm for finding such trees. 

We anticipate that Theorem \ref{thm:main_simp} will find many applications.
We present three examples, the first two make significant progress on two long-standing and central problems in induced Ramsey theory and in the theory of random graphs, while the third is a near tight result on induced forests in countable (very strong) expanders. 
% The first is in the theory of random graphs, where it was not even known (regardless of how large $d$ is  relative to $\Delta$) if there existed any $C_d$ such that whp $G(n,d/n)$ contains every tree of max degree $\Delta$ and order $C_d \cdot n$.
% See Section \ref{sec:induced_random} for an in-depth background and motivation. 

% The second application is in induced and size induced ramsey theory, where no linear upper bounds were known for the ramsey numbers of bounded degree trees (linear in the order of the tree).
% We resolve this and in fact show something stronger, a density universality result. 
% See Section \ref{sec:induced_ramsey} for an in-depth background and motivation. 
% \begin{theorem}
%     For all $\Delta,n\in \mathbb{N}$ and $\eps >0$ there exists a graph $G$ with less than $(10^9\Delta^2/\eps) n$ edges such that any subgraph of $G$ containing $\eps\cdot e(G)$ edges contains every tree of maximum degree $\Delta$ and order at most $n$ as an induced subgraph of $G$.
% \end{theorem}

\subsection{Induced Subgraphs of Random Graphs}
\label{sec:induced_random}
One of the oldest problems in the theory of random graphs is estimating the size of the largest independent set in $G(n,p)$.
This value was asymptotically determined for $0<p<1$ constant and $n\rightarrow \infty$ by Grimmett and Mcdiarmid in the 70's \cite{grimmett1975colouring}.
They showed that with high probability $G(n,p)$ contains an independent set of order 
\begin{equation}\label{eq:indepednent_grimmett}
    \frac{2+o(1)}{\log (1/(1-p))}\log n.
\end{equation}
The matching upper bound follows from the first moment method.
It is natural to wonder whether there is anything special about independent sets in this regard or whether any sufficiently sparse induced graph of the same order should also appear. 
Indeed, Erd\H{o}s and Palka \cite{erdos1983trees} showed in the early 80s that the order of the largest induced tree in $G(n,p)$ was also given by \eqref{eq:indepednent_grimmett}.
% Erd\H{o}s and Palka already asked whether their results could be extended to the sparse regime. 
% The first step towards this was 
If one wants find a specific tree on that many vertices, such as a path or regular tree,  then one can turn to Ruci\'nski \cite{rucinski1987induced} who proved the same bound \eqref{eq:indepednent_grimmett}, provided the maximum degree of the tree is sub-polynomial in its order.

Another regime of $G(n,p)$, the so called \emph{sparse} regime, where $pn = d$ is constant, has received a lot of attention in recent decades. 
This regime presents mathematical challenges not faced in the dense case (the classical second moment approach breaks down), and further, most applications coming from computer science, statistical physics and mathematical modelling use graphs of constant average degree.
Erd\H{o}s and Palka suggested the problem of extending their result to the sparse regime and conjectured that for all $d>1$ there exists $C_d$ such that $G(n,d/n)$ contains an induced tree of order $C_d\cdot n$ with high probability.
In other words they conjectured that there is an induced tree of linear order in $G(n,d/n)$. 
In a flurry of activity this was simultaneously proven by De la Vega \cite{de1986induced}, Frieze and Jackson \cite{frieze1987large}, Ku{\v{c}}era and R{\"o}dl \cite{kuvcera1987large}, 
and {\L}uczak and Palka \cite{luczak1988maximal}.
The best constant, $C_d = (1+o(1))\log d/d$, was due to De La Vega who proved that the greedy algorithm succeeds (with high probability) 
 by tracking stochastic differential equations.
Frieze and Jackson \cite{frieze1987largehole} soon showed that one can actually find an induced path (even a cycle) of linear order, albeit
 for a smaller constant $C_d$ and only if $d$ is sufficiently large.
 Suen \cite{suen1992large} and {\L}uczak \cite{luczak1991cycles}
 improved the result on induced paths to show that as long as $d>1$, $G(n,d/n)$ contains an induced path of linear order and further that for large $d$ one could take $C_d = (1+o(1)) \log d/d$, matching De La Vega's bound for induced trees.
 Suen's result was a particularly elegant use of the depth first search tree that avoided a lot of the technicalities of the stochastic differential equation method.

Of course, the fact that $G(n,d/n)$ contains an independent set of linear order follows from the average degree of the graph and a random greedy algorithm finds an independent set of order $(1+o(1))(\log d /d)n$ (matching the order of De le Vega's induced tree). 
To this day this is the largest independent set or induced tree one can find in $G(n,d/n)$ \emph{efficiently}. 
This value appears to be an algorithmic barrier related to the so-called shattering and freezing thresholds \cite{coja2015independent} and different tools are needed to go beyond it. 
Indeed, the asymptotic order of the largest independent set in $G(n,d/n)$ was not determined until the brilliant insight of Frieze \cite{frieze1990independence} who proved that if $d$ is sufficiently large then it is
\begin{equation}\label{eq:sparse_independence}
(2+o(1))\frac{\log d}{d}n,
\end{equation}
extending \eqref{eq:indepednent_grimmett}.
De La Vega, using Frieze's result as a black box, showed that the order of the largest induced tree is also \eqref{eq:sparse_independence}.
For all these exact asymptotic results, the upper bounds follow from the first moment method and 
the first moment gives the same bound for any fixed graph of degeneracy at most $2$, such as a matching, a tree or a cycle.
This suggests this bound may be tight for a broader class, provided, say $d/n < p < .99$ for some appropriate constant $d$.
However, following the rapid progression of the seventies, eighties and early nineties, the past three decades have been comparatively slow.
Dragani\'c \cite{Nemanja2020trees} extended Ruci\'nski's results beyond the constant $p$ regime by showing that $G(n,p)$ contains any bounded degree tree of order \eqref{eq:indepednent_grimmett}, provided $n^{-1/2}\log^{10/9}n<p<.99$.
He also conjectured that the result should apply to all trees of maximum degree $\Delta$ provided $d/n < p <.99$ for some $d(\Delta)$ (recall that \eqref{eq:indepednent_grimmett} and \eqref{eq:sparse_independence} coincide).
Cooley, Dragani\'c, Kang and Sudakov \cite{cooley2021large} showed that the order of the largest induced matchings is indeed given by \eqref{eq:indepednent_grimmett} and \eqref{eq:sparse_independence}, and they made the same conjecure.
Paths are perhaps the simplest case of the conjecture and it was only recently that Dragani\'c, Glock and Krivelevich \cite{draganic2022short} showed that, so long as $d$ is sufficiently large, $G(n,d/n)$ contains an induced path of order \eqref{eq:sparse_independence} with high probability.
They further re-iterated the conjecture of Dragani\'c.
But in spite of this precise conjecture, and the fact that we know  asymptotically the order of the longest induced cycle and the order of the largest induced tree, \emph{no linear bound} for a general bounded degree induced tree in $G(n,d/n)$ was known. 
We remedy this.
 
\begin{theorem}\label{thm: random_trees}
    There is $C>0$, such that for all $\Delta\in \mathbb{N}$ and $d> 2^{20\Delta}$, $G(n,d/n)$ contains all trees with maximum degree at most $\Delta$ and order at most $ \frac{Cn}{d\log^2(d)}$ as induced subgraphs with high probability. 
\end{theorem}
We observe our result is essentially tight as a function of $d$, up to a $C\log^3(d)$ factor. We also note that we could drop the lower bound on $d$ (to a linear function of $\Delta$) provided the order of the tree is at most $Cn/(d\log(d))^2$. 

\subsection{Ramsey Theory of Sparse Graphs}\label{sec:induced_ramsey}
One of the most famous recent results in Ramsey theory is the Burr-Erd\H{o}s Conjecture, proved by Lee \cite{lee2017ramsey}. 
It states that for all $d$ there exists $C_d$ such that any $n$-vertex graph with degeneracy at most $d$ has ramsey number  at most $C_d \cdot n$, in other words \emph{graphs of bounded degeneracy have linear ramsey numbers}. 
This extended the central result of Chvat\'al, R{\"o}dl, Szemer\'edi and Trotter \cite{chvatal1983ramsey}, who proved that bounded degree graphs have linear ramsey numbers. 
The original bound from \cite{chvatal1983ramsey} on $C_\Delta$ came from the regularity lemma and was thus huge. 
This was greatly improved by Graham, R{\"o}dl and Ruci\'nski \cite{graham2000graphs} and further by Conlon, Fox and Sudakov \cite{conlon2012two} who showed one can take $C_\Delta = 2^{c\Delta\log \Delta} $.
The best lower bound, coming from bipartite graphs is $2^{c\Delta}$ (also due to \cite{graham2000graphs}) and this is conjectured to be tight (up to the constant in the exponent) by Conlon, Fox and Sudakov \cite{conlon2015recent}.
If tight for all $\Delta(n)$ then this would given a nice generalisation of the upper bound of Erd\H{o}s and Szekeres (up to the constant in the exponent).

{Two natural generalisations of Ramsey theory are size ramsey theory $\rsi$ and induced Ramsey theory $\rind$, and we further have their common generalisation, size induced ramsey theory $\rsind$.
\begin{align*}
    \rind(H) & = \min\{v(G) :  \text{in any $2$-colouring of $E(G)$ there is a monochromatic copy of $H$ that is induced in $G$}\}, 
    \\
    \rsi(H) & = \min\{e(G) : \text{in any $2$-colouring of $E(G)$ there is a monochromatic copy of $H$}\},
    \\
    \rsind(H) & = \min\{e(G) :  \text{in any $2$-colouring of $E(G)$ there is a monochromatic copy of $H$ that is induced in $G$}\}.
\end{align*}}
% Induced paths and size induced paths proved by \cite{haxell1995induced}.
Note that trivially we have  $r(H) \leq \rind(H),\rsi(H) \leq \rsind(H)$, and further $\rsi(H) \leq {r(H)\choose 2}$ and $\rsind(H) \leq {\rind(H) \choose 2}$.
In all three of the above cases there are trees (which have degeneracy one) that have ramsey number superlinear in their order, thus the Burr Erd\H{o}s conjecture does not generalise. 
A fundamental question then asks:
\begin{quote}
    \textit{For which families of graphs are the ramsey numbers linear in the number of vertices?}
\end{quote}
Let $\TTT$ be the set of all trees and $\TTT_\Delta$ those of maximum degree at most $\Delta$, and let $\GG_\Delta$ and $\HH_d$ be the families of graphs of degree and degeneracy at most $\Delta$ and $d$ respectively. 
In Table \ref{ramsey-table} we collect estimates for the maximum ramsey number of an $n$-vertex graph from each family.

\begin{table}
\begin{tabular}{  >{\centering\arraybackslash}m{2.2em} | >{\centering\arraybackslash}m{1.4cm}|>{\centering\arraybackslash} m{1.3cm} |>{\centering\arraybackslash}m{3cm}| >{\centering\arraybackslash}m{4.5cm} |>{\centering\arraybackslash}m{3.2cm} } 

    & Paths
    & $\TTT_\Delta$ 
    & $\TTT$
    & $\GG_\Delta$ 
    & $\HH_d$ \\ 
    
   \hline
   $r(\cdot)$ 
   & $\Theta(n)$  
   & $\Theta(n)$  
   & $\Theta(n)$  
   & $\Theta(n)$ \cite{chvatal1983ramsey} 
   & $\Theta(n)$ \cite{lee2017ramsey}   
  \\
 \hline
    $\rsi(\cdot)$
    & $\Theta(n)$ \cite{beck1983size} & $\Theta(n)$ \cite{friedman1987expanding} 
    & ${n^2/4\leq \cdot\leq n^3\log^4 n}$
    \cite{beck1990size}, \cite{beck1990size}
    & {$cne^{c\sqrt{\log n}}< \cdot < n^{2-1/\Delta-o(1)}$} \cite{rodl2000size,tikhomirov2022bounded}, \cite{kohayakawa2011sparse}  &  $\Theta(n^2)$   
    \cite{beck1990size},\cite{lee2017ramsey}
    \\ 
   
   \hline
    $\rind(\cdot)$
    & $\Theta(n)$ 
    \cite{haxell1995induced}
    & \textcolor{red}{$\Theta(n)$} & ${\omega(n)\leq \cdot\leq n^2\log^2 n }$ 
   \cite{fox2008induced}, \cite{beck1990size}
    & $ \cdot < n^{O(\Delta)}$

    \cite{conlon2014extremal}
    &  $\omega(n) < \cdot < n^{O(d\log d)}$ \cite{fox2008induced},\cite{fox2008induced} 
\\
 \hline
    $\rsind(\cdot)$
    & $\Theta(n)$ \cite{haxell1995induced} & \textcolor{red}{$\Theta(n)$} 
    & ${n^2/4\leq \cdot\leq n^3\log^4 n }$
    \cite{beck1990size}, \cite{beck1990size}
    & {$cne^{c\sqrt{\log n}}<\cdot< n^{O(\Delta)}$}
    \cite{rodl2000size,tikhomirov2022bounded},\cite{conlon2014extremal} &  {$n^2/4<\cdot < n^{O(d\log d)}$} \cite{beck1990size},\cite{fox2008induced}  
\end{tabular}
 \caption{\label{ramsey-table}
    Bounds for the maximum ramsey number of an $n$-vertex graph from each family; paths, bounded degree trees, trees, bounded degree graphs, bounded degeneracy graphs.
    Original results are in \textcolor{red}{red}.}
\end{table}

For size ramsey numbers linear bounds for paths  were proven by \cite{beck1983size}.
This answered a question from the seminal paper of Erd\H{o}s, Faudree, Schelp and Rousseau \cite{erdHos1978size} for which Erd\H{o}s had later offered $\$100$ \cite{erdHos1981combinatorial}. 
With an impressive application of the the probabilistic method Beck \cite{beck1983size} also proved an upper bound of $C_\Delta n\log^{12}n$ for trees of maximum degree $\Delta$, while he conjectured that a bound of $C_\Delta n$ should hold.
Friedman and Pippenger's \cite{friedman1987expanding}  ``beautiful" result proved  a linear upper bound and this was tightened by Haxell and Kohayakawa \cite{haxell1995size} via a subtle anyalysis of Friedman and Pippenger's method, resolving Beck's conjecture\footnote{Beck actually made an even more specific conjecture for each tree, but we will not discuss it here.}. 
On the other hand it was shown that no such bounds are possible for general trees by Beck \cite{beck1990size}  or graphs of maximium degree $3$ by R{\"o}dl and Szemer\'edi \cite{rodl2000size}. 
R{\"o}dl and Szemer\'edi further conjectured that for all $\Delta$ there exists $\epsilon>0$ such that for all large $n$ the maximum size ramsey number of maximum degree $\Delta$ graphs on $n$ vertices is between $n^{1+\eps}$ and $n^{2-\eps}$. 
This upper bound was settled by Kohayakawa, R{\"o}dl, Schact and Szemer\'edi \cite{kohayakawa2011sparse} 
while for the lower bound the best result is due to Tikhomirov \cite{tikhomirov2022bounded} who significantly improved the bound of \cite{rodl2000size} through a clever random twist on their construction.

For induced ramsey numbers (and in fact size induced ramsey numbers), a linear bound for paths was proved by Haxell, Kohayakawa and {\L}uczak \cite{haxell1995induced}.
 Fox and Sudakov showed that no linear upper bound for induced ramsey numbers of trees was possible, while remarkably the case of bounded degree graphs remains wide open.
Indeed, no non-trivial (super-linear) lower bound is known while the best upper bound due to conlon, Fox and Zhao \cite{conlon2014extremal} is $n^{C\Delta}$. 
In \cite{fox2008induced} the authors asked if there exists a constant $C$, independent of $\Delta$, such that the induced ramsey number of $n$-vertex graphs with maximum degree at most $\Delta$ is at most a polynomial in $n$ of degree at most $C$ (the coefficients may depend on $\Delta$).
% We remark that Erd\H{o}s made a wonderful conjecture that there exists a constant $c>0$ such that all graphs $H$ of order $n$ satisfy $\rind(H) < 2^{cn}$. 

Of course lower bounds from size ramsey numbers yield lower bounds for induced size ramsey numbers, thus neither bounded degree graphs nor trees satisfy  linear upper bounds. 
As was remarked in the paper of Bradac, Dragani\'c and Sudakov \cite{bradavc2023effective}: "...for bounded degree trees we know that the size-Ramsey number is linear in their number of
vertices, whereas for its induced counterpart we have no good bounds while we have every reason to believe
that the answer should also be linear".
We prove the first linear bounds for induced and size induced ramsey numbers of general bounded degree trees. 

\begin{theorem}\label{thm: 2coloursizeramsey}
    For all $\Delta \in \mathbb{N}$ there exists $C_\Delta$ such that for any tree $T$ of maximum degree at most $\Delta$ on $n$ vertices $\rsind (T) < C_\Delta \cdot n$. One can take $C_\Delta = 10^{25}\Delta^3\log(\Delta)$.
\end{theorem}
Of course one can replace $\rsind$ by $\rind$ for free.
While this paper was in preparation Hunter and Sudakov~\cite{HunterSudakov} also proved the above theorem.
They proceed by very different techniques, cleverly reducing the problem to the non-induced case, via carefully constructed subgraphs of blowups.  
Their methods give worse bounds (an exponential dependency on $\Delta$) but one nice thing is that it extends to graphs of bounded treewidth. 
In order to do this they once again reduce to the non-induced case which was proved by \cite{berger2021size} (the two-colour case was proved simultaneously by \cite{kamcev2021size}). 
These results in turn use Friedman and Pippenger's result as a black box. 
Thus, it is very natural to ask if one can use our induced Friedman and Pippenger type result to avoid any reduction to the non-induced case and improve the quantitative bounds due to Hunter and Sudakov (which are very large due to the use of sparse regularity).

In fact, we prove something much stronger than the above statement. 
Given a family of graphs $\GG$ a \emph{Universal graph} for $\GG$ is a graph that contains all graphs of $\GG$ as subgraphs. 
The study of such objects goes back at least as far as Moon \cite{moon1965minimal}, with the central question being bounds on their size and order.
The induced question is also well studied \cite{chung1983universal} and now
it is even known that an induced universal graph of order $O(n)$ exists for the family of trees on $n$ vertices (even unbounded degrees) \cite{alstrup2017optimal}.
The constructions are far from random. 
If in any colouring of our graph $G$ we have a monochromatic (induced) universal graph then we say that our graph is (induced) partition universal. 
An induced such result (also going by the name adjacency labelling) for bounded degree graphs is what was actually proved in \cite{kohayakawa2011sparse}.
Of course a density universal result is stronger again, that is, a result that says: in any subgraph of density $\eps>0$ one finds all of said subgraphs. 
This is what we have proven; a density universal theorem for bounded degree trees of order $n$ with an upper bound \textit{on the number of edges} that is linear in $n$. 
\begin{theorem}\label{thm: sizeramsey}
    For all $\Delta,n\in \mathbb{N}$ and $\eps >0$ there exists a graph $G$ with less than $C(\Delta,\eps) \cdot n$ edges such that any subgraph $J\subset G$ containing $\eps\cdot e(G)$ edges contains every tree of maximum degree $\Delta$ and order at most $n$ as an induced subgraph of $G$. One can take $C(\Delta,\eps) = \left(10^{42}\Delta^3\log(\Delta)\log(\frac{1}{\varepsilon})^3\right)/\eps^2$.
\end{theorem}
We remark that Butler \cite{butler2009induced} showed that there is no induced universal graph of linear order for bounded degree graphs, thus one cannot replace bounded degree trees by general bounded degree graphs in the above theorem, even if one sets $\eps=1$ and replaces the bound on the number of edges by the same bound on the order.
It was also shown by Chung and Graham \cite{chung1983universal}, that even with $\eps=1$, if we remove the bounded degree condition the theorem above does not work\footnote{Further, a minor adjustment to the proof allows one to prove the above theorem with  a larger constant, for bounded degree forests (instead of trees), see the discussion of rolling back in the conclusion.}.
In order to see the strength of the constant, note that this implies state of the art bounds even in the case of the multicolour size ramsey number of induced paths (see Theorem \ref{thm: sizeramseyqcolours} and discussion).

\subsection{Dense Induced Forests in Countable Sparse Expanders}
Friedman and Pippenger's result also applies to countable graphs, where it is best phrased in terms of Cheeger's constant. 
For a countable graph $G$ we define the vertex Cheeger constant as 
\[
h_v(G) \coloneqq \inf_{X}\left\{\frac{|N(X)\backslash X|}{|X|}   \right\}, 
\]
where the infimum is over all finite  sets $X\subset V(G)$. 
Friedman and Pippenger's result (because it is online) implies that if $d>3$ and $h(G) \geq d$ then $G$ contains a tree with $h(T) \geq d-2$. 
Benjamini and Schramm \cite{benjamini1997every} proved the following stronger theorem about \emph{spanning} forests, which is actually a characterisation in the case of graphs with integer Cheeger constant. 
\begin{theorem}[Benjamini-Schramm \cite{benjamini1997every}]
    Suppose $d\geq 0$ is an integer and $G$ is a graph with $h(G)\geq d$. 
    Then $G$ has a rooted spanning forest in which the roots have degree $d$ and all other vertices have degree $d+2$.
\end{theorem}
One can also define the Cheeger constant in terms of edge boundaries (this is perhaps the more classical quanitity), 
\[
h(G) \coloneqq \inf_{X}\left\{\frac{e(X,G\backslash X)}{|X|}   \right\}, 
\]
where again the infimum is over finite sets $X\subset V(G)$. 
The use of the edge boundary is essentially forced on you if you are looking for induced structures.
For regular graphs $G$ this Cheeger constant is dual to the supremum of average degree over all finite subgraphs.
Our technique extends, as Friedman and Pippenger's did, to finding an induced tree in countable graphs and in fact it immediately yields a spanning version, in the spirit of Benjamini and Schramm. 
The attentive reader will correctly object that there is only one induced spanning subgraph of a graph.
Our result is not spanning but is \emph{as spanning as possible} (given the degrees of the induced forest we embed). 
A \emph{pseudoforest} is a graph in which every component contains at most one cycle. 
A $\Delta$-ary pseudoforest is a pseudoforest in which all acyclic components are $\Delta$-ary trees and all other components are $(\Delta+1)$-regular.
A subgraph is \emph{component-wise induced} if each component is an induced subgraph.

\begin{theorem}\label{thm:countable}
    There exists $\eps>0$ such that if 
    $G$ is a $d$-regular graph with $h(G) > d - 3 + \frac{1}{10^7\Delta+1}$ for some $\Delta < \eps d$,
    then $G$ contains a spanning $\Delta$-ary pseudoforest $F$, that is component-wise induced, with the property that one can turn $F$ into an induced forest by deleting one vertex from each of its components.
\end{theorem}

There is only one $d$-regular graph with $h(G)> d - 2$, and that is the $d$-regular tree, and if $h(G)=d-2$ then $G$ is a pseudoforest. 
Our result says that we can weaken this condition and still find a "dense" family of induced trees. 
The induced forest intersects all neighbourhoods in at least $\Delta$ vertices which is essentially best possible. 
Further, one cannot weaken the bound on Cheeger's constant by much (even without the component-wise induced condition), as is witnessed by blowing up the edges of the $d/2$-regular tree by $K_{2,2}$'s (one replaces the vertices with independent sets). 
This graph has $h(G) = d - 4$ but does not even contain a binary spanning pseudoforest with the properties described in Theorem \ref{thm:countable} 
\footnote{To see this we will refer to the vertices that are to be deleted in the statement of Theorem \ref{thm:countable} as roots of $F$.
If there is any pair $u$ and $v$ of twin vertices with $v$ not a root then $v$ must be in the same component of $F$ as $u$ (deleting $u$ leaves an induced tree containing neighbours of $u$). Call this component $C$. 
By the inducedness of $T=C-r$ (where $r$ is the root of $C$) all neighbours of $u$ in the tree  are neighbours of $v$ in the tree. 
It follows that there is at least one $4$-cycle in $C$ containing $u$ and $v$, but then the component must have been $3$-regular (by the definition of a binary pseudoforest) in which case there are three $4$-cycles, a contradiction.
Thus all pairs of twins are both roots, and so are all vertices, which is clearly impossible.}.

% One could also view the above result in terms of deleting vertices in order to delete all cycles. 
% Finding small such sets of vertices is the Feedback Vertex Set Problem, which is one of the oldest problems in computational complexity. 
% We also have the classical result of Erdos and Posa \cite{posa1965independent}.
% \begin{theorem}[Erdos-Posa]
%     For all $k\in \mathbb{N}$ there exists $K\in \mathbb{N}$ such that a finite graph $G$ either contains $k$ vertex disjoint cycles or can be made acyclic by deleting at most $K$ vertices. 
% \end{theorem}

% Of course it is not difficult to see that this theorem generalises immediately to countable graphs, but in an un-interesting way. 
% Is there someway we can allow infinitely many disjoint cycles but still be able to remove all cycles without deleting too many vertices? 
% Indeed we can.
% \begin{theorem}
%     There exists a universal constant $\eps > 0$ such that the following holds. Suppose $G$ is a graph with minimum degree $\Delta$ and maximum degree $1000\Delta$ such that all finite sets have average degree at most $C<3$. 
%     Then by deleting at most $(1-\eps)\deg(v)$ neighbours of any vertex $v\in V(G)$ we can delete all cycles from $G$.
% \end{theorem}

\subsection*{Notation}
Let $N(v)$ denote the open neighbourhood of $v$ and let $N[v] \coloneqq N(v) \cup v$ be the closed. 
For a set $X\subset V(G)$ we let $N(X) \coloneqq \cup_{v \in X}N(v)$ and $N[X] = X \cup_{v \in X}N(v) $.
% and $N^2(X) \coloneqq N(X) \cup_{v \in X}(N(N(v)) \backslash v)$.

\section{Proof Overview}
Our approach is inspired by the proof of Aggarwal, Bar-Noy, Coppersmith, Ramaswami, Schieber and Sudan \cite{aggarwal1996efficient},
who showed that a certain online matching game was winnable in robustly expanding bipartite graphs. 
The online matching game goes as follows. 
We have a graph $H=(A,B)$ and $X\subset A$ matched to a set $Y\subset B$ (initially these will be empty). 
The adversary picks an element $x\in A\backslash X$ and we have to choose an element $y\in B\backslash Y$ and match $X\cup x$ to $Y\cup y$.

\begin{theorem}\label{thm:bipartite}
    Let $H=(A,B)$ be a bipartite subgraph. Suppose that for all subgraphs $F\subset H$ with $\deg_F(v) > \deg_H(v)/2$ for all $v\in V(H)$ we have $|N_F(X)| > 2|X|$ for all $X\subset A$ with $|X| \leq n$.
    Then there is a polynomial time algorithm to find an online matching of order $n$ in $A$.
\end{theorem}

We call the high-level strategy used for Theorem \ref{thm:bipartite} and in this paper the Pre-emptive Greedy Algorithm.
It was applied to the problem of efficiently finding linear order bounded degree trees in expanders by \cite{dellamonica2008algorithmic}.
While \cite{dellamonica2008algorithmic} used Theorem \ref{thm:bipartite} as a black box, we describe the high level strategy if one were to open the box and run the argument in the case of bounded degree trees (\underline{not necessarily induced}). 
The high level strategy is the same in our case. 
The game is as follows.

\subsection*{The Game}
We wish to embed a bounded degree tree $T$ in a graph $G$. 
The graph $G$ is given to us, but $T$ is chosen by the adversary one vertex at a time. 
We start from the empty tree $T$, and in each round the adversary adds one vertex to $T$, to obtain $T'$, maintaining that $T'$ has both maximum degree at most $\Delta$ and is a tree. 
In each round we find an embedding of $T'$ in $G$ that extends our earlier embedding of $T$.
Formally we find an injective homomorphism $\phi'(T')$ such that the restriction of $\phi'$ to $T$ is simply $\phi$, the homomorphism we had from the previous round. 
Then in each round our adversary asks us to extend our current tree (embedded in $G$) from a vertex of degree at most $\Delta-1$ (in the embedded tree $T$). 
We lose when we cannot extend, we win if we play $n$ rounds without losing (and thus embed an $n$-vertex tree). 

We now \emph{informally} describe the strategy, leaving all formal definitions to later sections. 

\subsection*{The Pre-Emptive Greedy Algorithm}
Let us begin na\"ively.
If one were to greedily embed the tree, 
one could, after some short amount of time, end up in the following situation. 
The adversary asks you to extend the tree from a vertex $v$, but all of $v$'s neighbours are already in the tree. 
This prevents you from extending your tree without creating a cycle, so you lose. 
You must \emph{pre-empt} this situation. 
One thing you could try would be to watch all the vertices in the graph each time you extend your tree. 
If a vertex $v$ has too many neighbours in the tree i.e. it is \emph{critical}, then one immediately \emph{reserves} some neighbours of $v$ that are not in the tree. 
You \emph{reserve} these vertices for the eventuality that at some point in the future you are asked to extend your tree from $v$. 
You will not use the vertices \emph{unless} you are extending from $v$. 
Note that $v$ is not necessarily in your tree, and it may never be.\footnote{If one wants to tighten Theorem \ref{thm:main_simp} by reducing the factor of $\Delta$, then this wastefulness is one place to start.}
But if you are at $v$ then you can \emph{escape} using the reserved vertices. 

Of course, there is an issue here, there is a risk of a \emph{criticality cascade}. 
That is, when we reserve vertices for $v$ we may take neighbours away from a different almost critical vertex $u$, forcing us to reserve vertices for $u$. 
But doing that may mean that we make another vertex $w$ critical and so on, raising two issues.
\begin{enumerate}
    \item  Perhaps the criticality cascade consumes the whole graph (i.e. it makes every vertex critical). 
    \item Perhaps in reserving vertices for $u$ and $v$ we have used all the neighbours of some other vertex $w$.
\end{enumerate} 
Thus we must be more clever. 
We require expansion properties of $G$ to show that criticality cascades cannot be much larger than the current tree, this prevents problem (1). 
For the second problem, rather than fixing critical vertices one by one, each time we extend our \emph{tree} (not including reserving vertices) we find all the vertices that \emph{might} (we can't predict the future, but we can bound it) be caught in a criticality cascade and call this set $C$ (this is not too big!).
Because they are only \emph{at risk} of becoming critical, the vertices in $C$ are not critical yet and still have many neighbours that are not in the tree. 
If we require many more neighbours than $\Delta$ to be available, then perhaps we can find a way to reserve $\Delta-1$ neighbours for each vertex in $C$ \emph{simultaneously}. 
Aggarwal, Bar-Noy, Coppersmith, Ramaswami, Schieber and Sudan used a robust expansion condition, Hall's Theorem and augmenting paths to show these remedies work in the case of online matchings (the case of trees reduces to that case).  
None of those tools work in the case of induced trees however the high level strategy is similar.
\begin{enumerate}
    \item Extend the tree greedily until some set  $C$ of vertices is at risk of being caught in a \emph{criticality cascade}.
    \item Simultaneously \emph{reserve} neighours for each of the vertices in $C$.
    \item Consider these \emph{reserved} vertices as part of the extended ``tree" and repeat.
\end{enumerate}
To summarise our earlier analysis, in order for this algorithm to work it suffices that: 
\begin{enumerate}
    \item Criticality cascades cannot be much larger than the set that they start from (see Lemma \ref{prop:density_of_critical}).
    \item If $C$ is not too large and each vertex in $C$ has enough neighbours not in the tree, then we can \emph{simultaneously} reserve $\Delta-1$ vertices for each vertex of $C$ (see Lemma \ref{lem: ISM}).
\end{enumerate}

\begin{remark}
    In the above algorithm, we watch every vertex all the time, and actually ensure that we can always extend from any vertex that has not been extended $\Delta$ times yet. 
    Thus we can actually build a spanning forest, this is what allows Theorem \ref{thm:countable}.
\end{remark}

As previously mentioned none of the arguments that were used in the non-induced case work in the induced case.
Thus we must introduce some tools.
These will allow us to prove Lemmas \ref{prop:density_of_critical} and \ref{lem: ISM}, before proving the main theorem.

\subsection{Formal Machinery}
We begin by introducing directed graphs, because directing an edge from $u$ to $v$ will be useful for encoding that we have reserved $u$ for $v$.
An oriented subgraph $D$ of a simple graph $G$ is a subgraph $H\subset G$ along with an orientation for each of its edges $E(H)$. 
We view $D$ as a digraph living on the same vertex set as $G$.
For a vertex $v\in V(G)$, we define the \emph{in-neighbours} of $v$ as $\Nm_D(v) \coloneqq \{u : uv \in E(D)\}$.
The following two propositions will allow us to extend our tree or pseudoforest while only focusing on local information. 
In spite of their simplicity they are crucial to our proof. 
\begin{proposition}\label{prop:}
    A connected graph $T$ has an orientation such that every vertex has in-degree at most $1$ and at least one vertex has in-degree $0$, if and only if $T$ is a tree.  
\end{proposition}
\begin{proof}
    The if direction is straightforward because if $T$ is a tree then we can choose a root abitrarily and orient all edges away from the root. 

    For the only if direction, let $r$ be a vertex of in-degree $0$ and observe that due to the connectedness of $T$ we have a path $P_{ru}$ from $r$ to $u$ for any $u\in V(T)\backslash r$. 
    Because $r$ has in-degree $0$, the edge $rw$ incident to $r$ in $P_{ru}$ must be oriented away from $r$. 
    Because every other vertex has in-degree at most $1$ we see that all edges in $P_{ru}$ are oriented away from $r$. 
    Now suppose, for contradiction that there is a cycle $C$ in $T$, and let $u$ be a vertex in the cycle with shortest distance to $r$ (in the undirected sense). 
    By choosing appropriate paths from $r$ via $u$ to other vertices in $C$, we see that every path within $C$, with end-vertex $u$, must be oriented away from $u$. 
    But this is impossible in a cycle.
\end{proof}

\begin{proposition}\label{pseudoforest}
    A graph $F$ has an orientation such that every vertex has in-degree at most $1$ if and only if $F$ is a pseudo-forest. 
\end{proposition}
\begin{proof}
    For the if direction we orient each component individually. If a component $C$ is a tree then we are done by Proposition \ref{prop:} and if it contains one cycle then we cyclically orient the cycle and orient all other edges away from said cycle. 
    
    For the only if direction fix such an orientation and consider a component $C$ of $F$. 
    If $C$ has a vertex of in-degree $0$ then we are done by the previous proposition. 
    Otherwise assume, aiming for a contradiction that $C$ has two distinct cycles $S$ and $S'$.
    Clearly both cycles must be cyclically oriented. 
    Thus each vertex in $S$ has an in-edge that is in $E(S)$. 
    The same is true for $S'$ and it follows that all edges incident to $S$ and $S'$ that are not in $E(S)$ or $E(S')$ respectively, are oriented away from $S$ or $S'$ respectively. 
    In particular this means that  the cycles cannot intersect in any path (including the path that is just a vertex), because all edges of $S'$ with exactly one endvertex in a maximal path $P$ in the intersection would both be oriented away from $P$, implying $S'$ is not cyclically oriented.  
    Therefore by the connectedness of $C$ there exists a path with at least one edge from $S$ to $S'$ and so it must start oriented away from $S$ and finish oriented away from $S'$, implying that there is a vertex of in-degree $2$ on the path. 
\end{proof}

Of course we are interested not only in trees, but in induced trees. 
Thus we define the following object which combines the local-witness properties of Propositions \ref{prop:} and \ref{pseudoforest} with a subtle inducedness condition.
The next definition is key to our proof. 
Let $G$ be a graph and $D$ a (bi-)oriented subgraph of $G$. 
We define $\Vin(D)$ and $\Vout(D)$ to be the vertices with at least one in-neighbour or at least one out-neighbour respectively. 
% Further we call a subgraph $D' \subset D$ in-induced subgraph of $D$ if, $D[\Vin(D')] = D'[\Vin(D')]$.

\begin{definition}
    Given a graph $G$, an escape-way $D$ is an oriented subgraph of $G$ satisfying 
    that the in-degree (in $D$) of each vertex of $G$ is at most $1$ and such that if $x,y\in \Vin(D)$ and $xy\in E(G)$ then exactly one of $xy$ or $yx$ is in $E(D)$.
    The latter condition implies that $\Vin(D)$ induces a subgraph of $D$ in $G$.
\end{definition}
For example, an induced forest, oriented away from its roots is an escape-way. 
An induced cycle, oriented to be a directed cycle, is an escape-way.
Instructively, the following example is also an escape-way. 
Take an induced forest $T_1,\dots,T_t$ and orient each tree away from some root $r_1,\dots,r_t$ ($r_i\in T_i$).
Now add to $E(G)$  (but not to $D$), a clique containing $r_1,\dots,r_t$. 
See the discussion following Corolary \ref{cor:ISM} for more motivation. 
Note that subgraphs of escape-ways are escape-ways and if $D$ is an escape-way and $u \in N^+_D(v)$, we say that $D$ \emph{reserves $u$ for $v$}.

In our pre-emptive greedy algorithm we are interested in extending our escape-way chunk by chunk. 
We cannot use Hall's Theorem or the augmenting path properties that \cite{aggarwal1996efficient} used.
Thus we need to understand when two escape-ways are compatible.
While Proposition \ref{prop:} showed that certain components of escape-ways are trees, we also require inducedness.
Therefore it is tempting to say that if $u$ is reserved for $v$, then no neighbours of $u$ are allowed to  join the tree. 
This is almost correct, but its flaw is fatal; if we block all neighbours of $u$ then we can never extend from $u$!
The correct rule is that no neighbours of $u$ are allowed to  join the tree, \emph{unless we are extending from $u$.}
For a vertex $v\in V(G)$, we define the set of \emph{available neighbours} of $v$ as $A_{D}(v)\coloneqq N_G(v)\backslash ( \Vin (D-v)\cup N_D^{-}(v))$. Note that $A_{D}(v)\subset N_G(v)$.
A neighbour $u$ of $v$ in $G$ is therefore not available if either $u$ has an in-neighbour that is not $v$ \emph{or} $u$ is an in-neighbour of $v$.

\begin{definition}
Suppose $D$ is a (bi-)oriented subgraph of $G$. 
We define $K(D)$ to be the (bi-)oriented subgraph of $G$, obtained by starting from $D$  and adding, for all edges $uv \in E(D)$, all edges $vw$ for $w \in N_G(v)\setminus u$.
% We will slighty abuse notation and write $D^*[X]$ in place of $(D[X])^*$ for $X\subset V(G)$.
\end{definition}
Observe that if $D$ is an escape-way then $K(D)$ contains no bi-oriented edges. 
Further it is clear that if $D\subset D'$ then $K(D)\subset K(D')$.
We say an escape-way $D$ in $G$ \emph{agrees} with a (bi-)oriented subgraph $H$ of $G$ if the following hold: 
\begin{itemize}
    % \item If $y\in \Vin(H)$, $x\in \Vin(D)$ and $xy\in E(G)$ then $xy \in E(D)$,
    \item For all $x \in \Vin(H)\cap \Vin(D)$ we have $N^-_H(x) = N^-_D(x)$, and 
    \item For all $xy \in E(H)$, $yx \not\in E(D)$.
\end{itemize}

\begin{proposition}\label{prop:ISM_equiv}
 Given escape-ways $D$ and $D'$ in $G$ the following are equivalent: 
 \begin{enumerate}
      \item $D\cup D'$ is an escape-way in $G$,
     \item $D'$ agrees with $D$,
     \item For each $x\in V(G)$ we have $N^{+}_{D'}(x) \subset A_{K(D)}(x)$.
 \end{enumerate}
\end{proposition}
\begin{proof}
    First we show that $(1)$ implies $(2)$ by showing the contrapositive. 
    There are two reasons $(2)$ may not hold. If there exists $x \in \Vin(D') \cap \Vin(D)$ but $N^-_H(x) \not= N^-_D(x)$, then $x$ has in-degree at least $2$ in $D\cup D'$ meaning that $D\cup D'$ is not an escape-way. 
    If $xy \in E(D)$ and $yx \in E(D')$ then $D\cup D'$ has bi-oriented edges, meaning it is not an escape-way. 

    We also show that $(2)$ implies $(3)$ via the contrapositive. 
    Suppose there exists $y \in N^+_{D'}(x)$ and $y \not\in A_{K(D)}(x)$. There could be two reasons for $y \not\in A_{K(D)}(x)$.
    Suppose first $yx \in D$. Then, $xy\in E(D')$ and $yx \in E(D)$, hence $D'$ does not agrees with $D$. 
    If on the other hand $zy \in E(D)$ for some $z \not= x$ then  $y \in \Vin(D') \cap \Vin(D)$ but $N^-_H(x) \not= N^-_D(x)$.

    Finally we show that (3) implies (1), again via the contrapositive. 
    As $D$ and $D'$ are escape-ways there are only two reasons that $D\cup D'$ would not be an escape-way. 
    The first is that some vertex $y$ has in-degree at least $2$ in $D\cup D'$. 
    But then we must have two distinct edges $xy \in D$ and $zy \in D'$ which would imply that $y \in \Vin(K(D) - z)$ and so $y \not\in A_{K(D)}(z)$.
    The second is that there are $x \in \Vin(D)$ and $y \in \Vin (D')$ with $xy \in G$ but $xy,yx \not \in E(D)\cup E(D')$. 
    In this case $xy \in E(K(D))$ and so $y\not\in A_{K(D)}(z)$ for any $z \not= x$.
    
\end{proof}

Having established these technical properties we prepare to prove that criticality cascades cannot be too large. 
We will use the following bootstrap percolation process to capture all vertices that might be caught in a criticality cascade from $X$. 
\begin{definition}
    Given a set of vertices $X$ in a graph $G$, we define $C(X)\subset V(G)$, the \emph{$d$-critical set of $X$} as the terminal set of the following boostrap percolation process. 
    We let $X_0\coloneqq X$ and given $X_i$ we let $X_{i+1}$ be the union of $X_i$ with all vertices $v$ that have at least $d$ neighbours within distance at most $2$ from $X_{i}$ in $G\backslash v$ (where $x \in X_i$ has distance $0$).
\end{definition}
It follows immediately from the definition that if $X\subset X'$ then the $d$-critical set of $X'$ contains that of $X$. 
Further, the $(d-1)$-critical set of $X$ contains the $d$-critical set. 
The reason this definition is useful is because of its relationship to escape-ways.
Essentially it provides an upper bounds on how much the available neighbourhood can be reduced by fixing certain escape-ways.

% \begin{definition}
%     Given a tree $T$ in a graph $G$ we define the $d$-critical vertices $C(T)\subset G$ of $T$ by the following process. 
%     We begin by letting $X_0\subset V(G)$ be all non-leaf vertices of $T$. 
%     We then build a (bi-)oriented subgraph $B_0$ of $G$ by adding all directed edges of the form $uv$ for $u\in X_0$ and $v\in N(u)$. 
%     Let $D_0 = K(B_0)$ and let $X_{1}$ be the union of $X_0$ with all vertices $v$ saisfying $|A_{D_0}(v)| < d$.
%     We build the (bi-)oriented subgraph $B_1$ of $G$ 
%     We repeat this process so given $X_i,B_i$ and $D_i=K(B_i)$ we define $X_{i+1}$ to be the union of $X_i$ with all vertices $v$ satisfying $|A_{D_i}(v)| < d$.
%     We then let $B_{i+1}$ contain all directed edges of the form $uv$ for $u\in X_1$ and $v\in N(u)$.
%     We stop the process when $X_{j+1}=X_j$.
%     The critical vertices are then defined as $C(T) \coloneqq X_j\backslash X_0$.
% \end{definition}

\begin{proposition}\label{prop:critical_bounds_available}
    Let $X$ be a set of vertices in a graph $G$, and let $C(X)$ be the $d$-critical set of $X$. 
    Suppose $B$ is an escape-way with $\Vout(B) \subset C(X)$.
    Then $|A_{K(B)}(v)| \geq \deg_G(v) - d$ for all vertices $v \not\in C(X)$.
\end{proposition}
\begin{proof}
    Let $v\not\in C(X)$. 
    The crucial fact is that if $u \in N(v)$ has distance at least $3$ from $\Vout(B)$ in $G\backslash v$, then  $u\in A_{K(B)}(v)$.
    By the definition of $C(X)$ if $v \not\in C(X)$ then at most $d-1$ neighbours of $v$ have distance less than $3$ from $X$. 
    It follows that at most that many have distance less than $3$ from $\Vout(B)$ in $G\backslash v$.
    % Recall that by the definition of $C(T)$, every vertex $v\not \in X\cup C$ satisfies $|A_{K(B')}(v)| \geq d$ where $B'$ contains every directed edge of the form $uv$ for $u \in X\cup C$ and $v\in N(u)$. 
    % But $B\subset B'$ which implies $K(B)\subset K(B')$ and $A_{K(B')}(v) \subset A_{K(B)}(v)$ for all $v \in V(G)$. 
    % Thus the same inequality that held for $B'$ holds with respect to $B$.
\end{proof}

The following straightforward Lemma shows that large criticality cascades generate sets that are noticeably denser than trees or cycles.
This is useful because it implies they cannot be too large in graphs with no dense spots (see proof of main theorem).

\begin{lemma}\label{prop:density_of_critical}
    Let $X$ be a finite set of vertices in a graph $G$, and let $C(X)$ be the $d$-critical set of $X$.
    \begin{itemize}
        \item If $G[X]$ is connected and $|C(X)| \geq 2|X|$ then there exists a graph $H\subset G$ on at most $(2d+2)|X|$ vertices with average degree at least $2 + \frac{d-2}{2d+2}$.
        \item If $C(X)$ is unbounded then there exist a sequence of finite graphs $H\subset G$ with average degree approaching $3(1-\frac{1}{d+1})$.
    \end{itemize}

\end{lemma}
\begin{proof}
    We analyse the bootstrap percolation that generates $C(X)$ from $X$, with the added assumption that the vertices are added one by one. 
    We break ties by a global ordering of the vertices.
    This does not affect the terminal set which we call $C\coloneqq C(X)$.
    We start from $X$ and let $v_1$ be the first vertex added and so on.
    By definition $v_{i+1}$ was added because it has at least $d$ neighbours who are at distance at most $2$ from $X \cup \{v_1,\dots,v_i\}$.
    
    For the first case let $C' = X \cup \{v_1,\dots,v_{|X|}\}$. 
    This exists by assumption.
    Let the graph $H'\subset G$ be a minimal spanning subgraph such that the boostrap percolation w.r.t. $H'$ starting at $ X$ reaches all of $C'$ and in the same order as above.
    We define four groups of vertices in $H$. The first two are natural, the initial set $X$ and the newly critical $C'\backslash X$. 
    Now for each $v\in C'\backslash X$, we know that there at least $d$ neighbours that caused $v$ to be in $C'$.
    Choose $d$ of these arbitrarily and call this set $Y_v$. 
    We then let $Y$ be the union of all these sets over $v\in C'\backslash X$. 
    Note that $Y$ may intersect both $X$ and $C'\backslash X$.
    Finally we define $Z \coloneqq V(H') \backslash ( C' \cup Y)$. 
    The role of the vertices in $Z$ (by process of elimination, and the minimality of $H'$) is to guarantee that the vertices in $Y$ are at distance less than $2$ from the earlier stages of the percolation.  
    % Clearly each $v\in C'\backslash X$ has degree at least $d$ in $H'$.
    We now construct a surjective map $f:Y \rightarrow Z$ in order to show $|Z| \leq |Y|$.
    We assume $Z$ is non-empty as the inequality is trivial otherwise.
    For $w \in Y$ let $j$ be the smallest index such that $w\in Y_{v_j}$.
    Let $f(w)$ be an arbitrary vertex in $Z$ on a path of length $2$ between $w$ and $X\cup \{v_1,\dots,v_{j-1}\}$ in $H'$.
    If no such vertex exists let $f(w)$ be arbitrary in $Z$. 
    By the minimality of $H'$ this map is a surjection. 
    Indeed otherwise there would exist a vertex $u\in Z$ whose deletion does not affect the bootstrap percolation rule. 
    As $Y$ has order at most $d(|C'|-|X|)$
    this implies that the order of $H'$ is at most $|C'| + (2d)(|C'|-|X|) = (2d+2)|X|$.

    We now iteratively delete leaves (degree one vertices) of $H'$ until we are left with a graph $H$ that has no leaves. 
     Observe that for each $v \in C'\backslash X$, every pair of edges incident to $v$ in $H'$ lie on a common cycle in $H'$. 
    This is because they both lie on paths from $v$ to $X$ and $G[X]$ is connected. 
    Therefore every edge incident to $v$ in $H'$ is also in $H$. 
    This allows us to lower bound the degree of every vertex in $H$ by $2$ and every vertex in $C'\backslash X$ (in $H$) by $d$. 
    By summing up the degrees in $H$ we obtain a lower bound on the average degree. In particular
    \begin{align}
      \frac{\sum \deg(v)}{v(H)} &\geq   \frac{d(|C'|-|X|) + 2(v(H)-(|C'|-|X|))}{v(H)} =  \frac{2v(H) + (d-2)|X|}{v(H)} 
       \\
       &= 2 + \frac{|X|(d-2)}{v(H)} \geq 2 + \frac{|X|(d-2)}{(2d+2)|X|} = 2 + \frac{d-2}{2d+2},
    \end{align}
where in the second step we used that $|C'|-|X| = |X|$. 
Thus $H$ is the desired graph.   

For the unbounded case, we can be more wasteful.
We construct a sequence of graphs $H_0,H_1,\dots$ with $H_0 = G[X]$ inductively as follows. 
Given $H_{i}$ if $v_{i+1}\in H_{i}$ then let $H_{i+1} = H_{i}$. 
Otherwise do the following. 
By the definition of $v_{i+1}$ it has at least $d$ neighbours $u_1,\dots,u_d$ that are at distance $0,1$ or $2$ from $H_{i}$.
We construct $H_{i+1}$ in $j$ steps $H_{i}=H_0\cup v_{i+1},\dots,H_i^d=H_{i+1}$, again inductively.
For $j\in [d]$ note that $\dist(u_j,H_i^{j-1})\leq 2$ by assumption. 
Add a path of length  $dist(u_j,H_i^{j-1})$ to $H_i^{j-1}$ to obtain $H_i^j$.
Thus we have added $dist(u_j,H_i^{j-1})$ vertices and $dist(u_j,H_i^{j-1})+1$ edges. 
Letting $M$ denote the sum over $j$ of the aforementioned distances we have that to get from $H_i$ to $H_{i+1}$ we have added $M + 1$ vertices and $M+d$ edges. 
Thus the ratio $(M+d)/(M+1)$ of edges added to vertices added is minimised when $M$ is maximised. 
$M$ is bounded above by $2d$ and so we obtain that the ratio of edges added to vertices added is at least 
\[
\frac{3d}{2d+1} > \frac32 \left(1-\frac{1}{d+1}\right).
\] 
Thus if $C(X)$ is unbounded then the average degree of the $H_i$'s approaches $3d/(d+1)$.
\end{proof}

Our final lemma before we prove the main theorem states that in graphs of small degeneracy we can find escape-ways that reserve a large portion of each vertex' neighbourhood for each vertex.
In fact it has a few more bells and whistles than that so we first state a corollary that captures its essence.

\begin{corollary}\label{cor:ISM}
    Let $G$ be a graph with maximum degree $\Delta$ such that every subgraph of $G$ has average degree at most $3$.
    There exists an escape-way $D$ in $G$ with $\dego_D(v)\geq \frac{|\deg(v)|}{10^7} - 5\log \Delta$ for all $v \in  V(G)$.
\end{corollary}

This result sheds light on the delicate definition of an escape-way. 
We have already seen in Propositions \ref{prop:} and \ref{pseudoforest} that the definition is sufficient to encode trees and pseudoforests. 
We now show that very slight alterations to the definition of an escape-way make Corollary \ref{cor:ISM} impossible. 

First suppose we required $\Vin(D)\cup\Vout(D)$ (instead of just $\Vin(D)$) to induce a subgraph of $D$ in $G$. 
We construct $G$ as follows.
 Consider a cycle $C$ of length $\ell$ with a single chord and add $d$ pendant edges to each vertex of the cycle.
 Then $G$ has maximum degree $d+3$ and every subgraph has average degree less than $3$.
If $d$ is large enough then every vertex of the cycle must be in $\Vout(D)$ for any $D$ as in Corollary \ref{cor:ISM}.
But if $\Vin(D)\cup\Vout(D)$ induced a subgraph of $D$ in $G$ then all of the edges in $C$ as well as the chord are in $D$.
But orienting these edges will create a vertex of in-degree $2$, which contradicts that $D$ is an escape-way.

Secondly suppose that instead of requiring that $\Vin(D)$ induces a subgraph of $D$ in $G$ we required that our escape-ways are unions of disjoint stars (oriented away from the root). 
Let $G$ be the $d$-ary tree of depth $3$. 
Thus $G$ has max degree $d+1$ and all subgraphs have average degree less than $2$.  
If $d$ is large then root $r$ must have at least one out-neighbour, say $v$. 
But then if $v$ also has an out-neighbour $\Vin(D)$ the stars (one rooted $r$ and one at $v$) are not disjoint.
Thus if $d$ is large, there is no $D$ satisfying Corollary \ref{cor:ISM}.

The  maximum degree condition in Corollary \ref{cor:ISM} is due to our use of the Lovász Local Lemma.
\begin{theorem}[Lovász Local Lemma]
    Suppose there are a set of events such that each event is mutually independent of all but $d$ other events. 
    If each event occurs with probability less than $1/ed$ then the probability that none of the events occur is positive.
\end{theorem}

We will also require McDiarmid's Inequality for Lipschitz functions of independent random variables. 
\begin{lemma}
    Suppose $X_1,\dots,X_m$ are independent binary random variables and $f:\{0,1\}^m\rightarrow \RR$ is an $s$-Lipschitz function. 
    Then for any $t>0$
    \[
       \Prob\left [ \mathbb{E}[X] - X > t \right ] < \exp\left(-\frac{2t^2}{m s^2}\right).
    \]
\end{lemma}

We now state the lemma in full detail. 
Given a subgraph $G'\subset G$ we denote by $A_D^{G'}(v) \coloneqq A_D(v) \cap N_{G'}(v)$, the available neighbours of $v$ in $G$ with respect to $D$, that are also neighbours of $v$ in $G'$.

\begin{lemma}\label{lem: ISM}
    Suppose $F$ is an undirected graph with  maximum degree $\Delta$ such that all subgraphs have average degree at most $3$.
    Further suppose $F$ has a spanning subgraph $G$ and an oriented subgraph $H$.
    Then $F$ has an escape-way $D$ that agrees with $H$, such that $\dego_D(v) \geq \frac{|A_H^G(v)|}{10^7} - 5\log \Delta $ for all $v \in  V$.
\end{lemma}
The reason for two graphs $F$ and $G$ is to prepare for our ramsey theoretic applications ($G$ will be a large monochromatic subgraph).
We cannot just forget about $F$ because in induced ramsey theory one wants a monochromatic subgraph that is induced in the original graph.
For the oriented subgraph $H$, the motivation is to encode the contraint that the escape-way $D$ that we seek should combine with an escape-way $D'$ that we already have, to give a bigger escape-way (see proof of main theorem). 

\begin{proof}
    The idea is to use the average degree condition  to bound the degeneracy by $C$ and find an orientation of $F$ in which all but $C$  of the available neighbours of each vertex are out-neighbours. 
    We then sample a random subgraph of this orientation and \underline{carefully} resolve clashes (keeping the earlier edge in the degeneracy ordering) to turn it into an escape-way called $\bD$.
    In expectation this will leave everyone with a large out-neighbourhood and we complete the proof by showing concentration (once again using the average degree condition) and applying the LLL.

    Fix an ordering of $V(G)$ which witnesses that  the degeneracy of $F$ is at most $C$. 
    Namely, writing $V(J)=V(G)=\{v_1,\dots,v_n\}$ we have  $|N(v_j)\cap \{v_1,\dots,v_{j-1}\}| \leq C$ for all $2 \leq j \leq n$. 
    Orient all the edges in $E(J)\backslash E(H)$  so that for all such edges $v_iv_j$ we have $i<j$, and copy the orientation of all other edges from $H$.
    Call the resulting digraph $G''$.
    Let $G'$ be the subdigraph of $G''$ consisting of edges of the form $vu$ for $u \in A_H^G(v)$ and observe that $|\No_{G'}(v)| \geq |A_H^G(v)|  - C$ for $v\in V(G)$.
 
    Now we choose a subdigraph $\bJ$ of $G'$ by including each edge of $G'$ independently with probability $p \coloneqq 1/C^2$.
    We then deterministically resolve clashes,  choosing a further subdigraph $\bD$ of $\bJ$ as follows. 
    For each $i \in [n]$, if: \begin{enumerate}
        \item $\degi_\bJ(v_i)\geq 2$ or
        \item  $\degi_\bJ(v_i)= 1$ and there exists $u\in \Ni_{G''}(v_i)\backslash \Ni_\bJ(v_i)$ with $\degi_\bJ(u) \geq 1$,
    \end{enumerate}  then we delete all edges of $\bJ$ that are oriented towards $v_i$.
    We observe that $\bD$ is a deterministic function of $\bJ$, and that by construction $\bD$ is an escape-way. 
    Indeed, $(1)$ ensures that the indegree of any vertex is at most $1$ and $(2)$ ensures that the vertices of indegree at least $1$ induce (in $F$) a subgraph of $\bD$.
    We denote  by $\bJ_{ww'}$ the event that $\{ww' \in E(\bJ)\}$, for each edge $ww'\in E(G')$.  
     Thus $\bJ$ can  be viewed as $\prod_{e\in E(G')}\bJ_e$.
     Similarly, we denote, for each edge $ww'\in E(G')$,  by $\bD_{ww'}$, the event that $\{ww' \in E(\bD)\}$.

    % First note that with probability one, we can obtain an ISM from $\bD$ by setting $\bR(v) = \No_H(v)$ for each $v$. 
    % Indeed $(1)$: that the sets $\{ \bR(u)\}_{u\in X}$ are pairwise disjoint follows from $(1)$ in the previous paragraph, $(2)$: that all edges induced by $\cup_{v\in X}\bR(v)$ in $G$ are of the form $vu$ for some $u\in \bR(v)$, follows from $(2)$ in the previous paragraph and $(3)$: that $u\in \bR(v)$ implies $v\not\in \bR(u)$ follows from the fact that $H$ is a subgraph of $G'$ which is an orientation of a simple graph. 
    
\begin{claim}\label{claim:expected_deg}
    For each $vu \in E(G')$ we have $\Prob[\bD_{vu}] \geq 1/(eC^2)$.
\end{claim}
\begin{proof}
    Consider an arbitrary edge $vu \in E(G')$.
    The probability that $vu \in E(\bJ)$ is simply $\Prob[\bJ_{vu}] = p$. 
    Then the probability  of $\bD_{vu}$ given $\bJ_{vu}$ is precisely the probability that for all $w \in \Ni_{G'}(u)\backslash v$ we have $wu \not\in E(\bJ)$ (i.e. not $\bJ_{wu}$) and $\degi_\bJ(w) = 0$.
    In other words, the probability that a set of at most $C(C-1)$ edges that appeared in $G'$ do not appear in $\bJ$.
    This occurs with probability at least $(1-p)^{C(C-1)}$.
    Recalling that  $p= 1/C^2$, and using the relation that $1-1/x \geq 
    \mathrm{e}^{-(1/(x-1))}$ which holds for all reals $x>1$, we can lower bound this probability by $\mathrm{e}^{-\frac{C(C-1)}{C^2-1}} \geq 1/\mathrm{e}$.
    Thus the probability of the event $\bD_{vu}$ is at least $p(1-p)^{C(C-1)} \geq 1/(eC^2) $.
\end{proof}

\begin{claim}\label{claim:independent_neighbours}
    For each vertex $v \in V(G)$ and edge $e\in E(G')$, changing the outcome of $\bJ_{e}$ effects the outcome of at most $8$ events $\bD_{vu}$.
\end{claim}
     \begin{proof}
     Recall that $\bD$ is a deterministic function of $\bJ$. 
     The event $\bD_{vu}$ depends on only: the event $\bJ_{vu}$, and for each $w \in \Ni_{G'}(u)\backslash u$ the events $\bJ_{wu}$ and $\bJ_{w'w}$ for each $w' \in \Ni_{G'}(w)$.
     The same can be said for $\bD_{e}$ for any edge $e \in E(G')$. 
     Thus if  $\bD_{vu}$ depends on $\bJ_{qw}$ then it must be that
     \begin{itemize}
         \item $qw=vu$ or
         \item $w=u$ or 
         \item $w\neq v$ and  $wu\in E(G')$.
     \end{itemize} 
     The first two candidates can hold for at most one vertex $u$ in total. 
     Thus if $\bJ_{qw}$ effects $\ell$ events $\bD_{vu}$ then there must be at least $\ell-1$ vertices in the common neighbourhood of $w$ and $v$. 
     But this implies that there exists a small subgraph of average degree $\frac{4(\ell-1)}{(\ell-1)+2}$.
     Thus by the condition that all subgraphs have average degree at most $3$, we have $\ell \leq 8$.
\end{proof}

\begin{claim}\label{claim:degree_concentration}
    For all $v \in V(G')$ we have 
    \[
    \Prob\left[\dego_\bD(v) < \frac{|A_H^G(v)|-C}{2eC^2}\right] < \exp\left[-\frac{(|A_H^G(v)|-C)^2}{|A_H^G(v)|\cdot 200\cdot e^2C^6}\right].
    \]
\end{claim}
\begin{proof}
    Let $v \in V(G')$ be arbitrary.
    As $|\No_{G'}(v)| \geq |A_H^G(v)|  - C$ we have, by Claim \ref{claim:expected_deg} that $\E[|\dego_\bD(v)|] \geq (|A_H^G(v)| - C)/(eC^2)$. 
    Further $\dego_\bD(v)$ is a function of $\{\bJ_{vu}\}$ where $u$ ranges over $\No_{G'}(v)$. 
    Recall that $\dego_\bD(v)$ depends on at most $m\coloneqq |A_H^G(v)| C^2$ events.  
    Further by Claim \ref{claim:independent_neighbours} it depends on these events in a Lipschitz manner with constant $8$.
    Therefore we can apply McDiarmid's inequality with $t = \frac{|A_H^G(v)|-C}{2eC^2}$, $m=|A_H^G(v)|C^2$ and $s = 8$ to get
    % variable $\bX_v \sim \text{Bin}(|A_H^G(v)|-C,1/(eC^2))$ to see that for all $t\geq 0$
    % \[
    % \Prob\left[\dego_\bD(v) < t \right] \leq \Prob\left[ \bX_v < t \right].
    % \]
    % Using the standard Chernoff bound, that for $\delta \in (0,1)$
    % \[
    % \Prob[\bX_v \leq (1-\delta)\E[\bX_v]]\leq \exp\left[-\frac{\delta^2\E[\bX_v]}{2}\right]
    % \]
    % and setting $\delta = 1/2$, we see that 
    \[
    \Prob\left[\dego_\bD(v) < \frac{|A_H^G(v)|-C}{2eC^2}\right] < \exp\left[-\frac{2\cdot t^2}{m\cdot s^2}\right]  \leq  \exp\left[-\frac{2(|A_H^G(v)|-C)^2}{|A_H^G(v)|C^2\cdot 8^2\cdot 4e^2C^4}\right].
    \]
    \end{proof}
    Call the quantity inside the exponent of the end last equations $q_v$. 
    For the final step we wish to apply the LLL to conclude that with positive probabilty we have that $\bD$
    satisfies the conditions of the Lemma. 
    We note that if $q_v \leq 5\log \Delta$ then the statement holds trivially for $v$ as $\dego_D(v)\geq 0 $.
    For all other $v \in V(G)$ we will introduce a bad event $\bB_v\coloneqq\{\dego_\bD(v) < \frac{|A_H^G(v)|-C}{2eC^2}\}$.
    Clearly if no bad events occur then the guarantees of the Lemma are satisfied. 
    We have, by Claim \ref{claim:degree_concentration}, that $\Prob[\bB_v] < \exp\left[-q_v\right]$ and by assumption $q_v > 5 \log \Delta$.
    It follows that $\Prob[\bB_v] <1/\Delta^5$. 
    
    Now we upper bound the number of dependent bad events per bad event.
    If for some pair $u,v \in V(G')$, we have that $\bB_u$ and $\bB_v$ are dependent then there must be a vertex $w$ at distance at most $2$ from both $u$ and $v$ in $J$.
    Thus for each $u$ there are at most $\Delta^4$ vertices $v$ such that $\bB_u$ and $\bB_v$ are dependent.
    We apply the LLL via the relation
    $e\cdot q\cdot d \leq \mathrm{e} \cdot  \frac{1}{\Delta^5} \cdot \Delta^4 = \frac{\mathrm{e}}{\Delta} \leq 1$.
    The LLL then tells us that with positive probability none of the events $\bB_v$ occur. 
    Further, as observed earlier, with probability one, $\bD$ is an escape-way, and it follows that there exists at least one escape-way $D$ as guaranteed by the Lemma. 
    Substituting $C=3$ completes the proof with the estimate $q_v > |A_H^G(v)|/10^7$.
\end{proof}

% \begin{remark}
%     One could prove a similar Lemma to Lemma \ref{lem: ISM} above by replacing the girth condition by the condition that small subgraphs have average degree at most $12/5$ (as we have in the main results).
%     However, one would lose the independence of Claim \ref{claim:independent_neighbours} and would therefore have to rely on a Talagrand type concentration inequality that allowed for small effects, instead of using the vanilla Chernoff bound. 
% \end{remark}

We are now ready to prove the following result which is a stronger version of our main result (Theorem~\ref{thm:main_simp}). 
\begin{theorem}\label{thm: main_simplev2}
    Let $\Delta \in \mathbb{N}$ and suppose $G$ is a graph of maximum degree at most $\exp(\Delta/10^9)$ such that all subgraphs on at most $(10^7\Delta+1) n$ vertices have average degree at most $12/5$. 
    Further suppose $J\subset G$ is a spanning subgraph with minimum degree at least $10^7\Delta$.
    Then $J$ contains all trees on $n$ vertices with maximum degree at most $\Delta$ as induced subgraphs of $G$.
\end{theorem}

\begin{proof}

% First note that each vertex in $J$ is contained in at most ${9\choose 3}$ copies of $C_4$. 
% Indeed, if $v$ is contained in $4$ $C_4$'s and let $H$ be the subgraph formed by their union.  
% We can view $H$ as being constructed as follows; start from $C_4$ containing $v$ and now repeatedly add $i$ edges and $i-1$ vertices for some $i\in [4]$. 
% Thus the number of edges is at least $4 + (v(H)-4)4/3$ and the ratio of edges to vertices is at least $(4/3) - (4/3v(H))$. 
% As $H$ contains at most ${v(H)-1\choose 3}$ copies $C_4$'s containing $3$ we have that $v(H)\geq 10$ and so $H$ has average degree at least $6/5$, which is a contradiction. 
% Thus there are at most $v(J){9\choose 3}/4 = 21 v(J)$ copies of $C_4$ in $J$.
% Now suppose $X$ is chosen by including each vertex of $J$ uniformly at random with probability $p=1/100 < {9\choose 3}^{-1}$. 

Let $T$ be a fixed tree on $n$ vertices and maximum degree $\Delta$. 
We pick a root $r\in V(G)$.  
We will follow an online process extending locally our current tree $T_i$ whereby at each step an adversary picks a vertex $x$ of degree at most $\Delta -1$ from the current tree and we extend the tree from that $x$.

\textit{The Pre-emptive Greedy Algorithm.}
Throughout the process we will maintain an escape-way $B_i$ which contains $T_i$ rooted at $r$.
We will write $D_i \coloneqq K(B_i)$.
On a high level, one can think of $B_i$ as those edges which are reserved for (possibly) extending the current tree $T_i$.
The directed-ness of an edge in $B_i$, say from $u$ to $v$, is to signal the fact that \emph{if} the edge $uv$ is ever used, then that is because we extend \emph{from} $u$ \emph{to} $v$.

In step one our adversary chooses a root $r = T_1\subset T$ and we choose any vertex $r \in V(G)$ to be the root. 
The root requires some special treatment (an unimportant technicality).
We add all edges from $r$ to $N_G(r)$ into $B_1$ and we delete all other vertices adjacent to $r$ from $J$. 
Due to the condition that there are no small graphs of average degree at least $12/5$ co-degrees are bounded $7$ and this deletion reduces the degree of any vertex in our graph by at most $7$. 
For convenience and because $d \geq 1000$, we shall ignore this $7$ and continue to call the remaining graph $J$.
We now state the formal properties to be maintained throughout the whole process.  
Suppose at step $i$, we have  a tree $T_i$ in $G$ and an escape-way $B_i$ in \emph{both} $G$ and $J$. 
We will always assume there is an isomorphism $f$ from $T_i$ to a sub-tree of $T$ and for ease of notation identify $x\in T_i$ with $f(x)$. 
We let $X_i$ denote the non-leaf vertices of $T_i$  and let $C_i \coloneqq C(X_i)\subseteq V(J)$ denote the $d$\textit{-critical} vertices of $X_i$.
Note that $X_i$, $C_i$ and $D_i$ are determined by $T_i$ and $B_i$.
% with  $|A^G_{D_i}(v)| < d_G(v) - d$.
We will have $B_i\subseteq B_{i+1}$ and $T_i\subset T_{i+1}$, and thus $X_i\subset X_{i+1}$, $D_i\subseteq D_{i+1}$ and $C_i\subseteq C_{i+1}$ for all $i\in [n-1]$. 
We further ensure:
\begin{enumerate}  
    \item\label{item:tree_is_sub} $T_i\subset B_i$;
    \item\label{item:critical_is_out} $\Vout(B_i) = C_i$;
 \item\label{item:reserved_enough}
    For every $v \in C_i$ we have $\deg^{+}_{B_i}(v) \geq \Delta -1$ and $\deg^{}_{B_i}(r) =\deg_G(r)$.
    % \item \label{item:indegree of B} For every $v \in \Vin(B_i)$ then $\deg^{-}_{D_i}(v)=1$; 
    % \item \label{item:oriented away from B} For every $v \in \Vin(B_i)\cup r$ all edges incident to $v$ but one (or $0$ if $v=r$), are oriented outwards from $v$, that is $\deg^{+}_{D_i}(v) = \deg_J(v) -1$;
    % \item \label{item: critical} $C_{i}\subseteq \Vout(B_i)$;
    % \item \label{item: extendability} The \emph{extendability condition}: for every $v \notin C_i$, then $  |A^G_{D_i}(v)| > d_G(v) - d$.
\end{enumerate}
As $C_1 = r$ it is trivial to see that the conditions are satisfied initially.
We now proceed by induction and assume we have $T_i,B_i,X_i,C_i$ and $D_i$ as described for some $i \in [n-1]$.
At each step, an adversary chooses a vertex $x$ from the current tree $T_i$ with $d_{T_i}(x)<\Delta$ and asks us to append a leaf to $x$. 
We now show how we can maintain the above properties while extending $T_i$.
Let $x\in T_i$ be the chosen vertex.
\\ 

\subsection{Case $1$: $x\in C_i$.}
We simply choose an out-neighbour of $x$, say $y\in N^+_{B_i}(x)\backslash N_{T_i}(x)$. 
This exists by \eqref{item:reserved_enough} since $d_{T_i}(x) \leq \Delta -1 $ and because if $x$ is not the root then it has an in-neighbour in $B_i$.
We let $T_{i+1} = T_i \cup xy$. We then define $B_{i+1}\coloneqq B_i$ and note that it contains $T_{i+1}$.
\\

\subsection{Case $2$: $x\notin C_i$.}
In this case, $x$ has many available neighbours, and so we would like to simply add them to our tree.
However, informally, adding arbitrary edges to $B_i$ may trigger a criticality cascade.
In order to find out which vertices near $x$ might be so affected we look at the $d$-critical set $C_{i+1} \coloneqq C(C_i \cup x)$ of $C_i\cup x$.
Because $C_i = C(X_i)$ we have that $C_{i+1} = C(X_i\cup x)$.

We observe that $|C_{i+1}|< 2|X_i\cup x|$. 
Indeed suppose otherwise, then the $100$-critical set of $X_i\cup x$ is at least as large and so  by Lemma \ref{prop:density_of_critical}, there would exists a graph $H\subset J$ of order at most $300n$ and average degree at least $12/5$. 
This would contradict the assumptions of the theorem. 
Thus, in particular, we have that $C^* \coloneqq C_{i+1} \backslash C_i$ has order most $n$.
Further, because $\Vout(B_i) = C_i$ by induction, Proposition \ref{prop:critical_bounds_available} tells us that $|A^G_{D_i}(v)| \geq \deg_G(v) - d \geq d$ for all $v\not\in C_i$ (and thus all $v\in C^*)$.

For each $v\in C^*$ we add $d$ vertices from $A^G_{D_i}(v)$, along with $v$, to a set we call $C^+$. 
Thus $|C^+| \leq (d+1)|C^*| \leq (d+1)n$.
We now give as input to Lemma~\ref{lem: ISM}, the graph $J[C^+]$ as $F$, the graph $G[C^+]$ as $G$ and $D_i$ the (bi)oriented subgraph of $J$ as $H$.
We input the maximum degree of $\exp(d/10^9)$ and note that this $F$ has degeneracy at most $3$, because it has order most $10^7\Delta n$ (degeneracy is at most the maximum average degree over all subgraphs).
Lemma~\ref{lem: ISM} therefore gives us an escape-way $I$, that agrees with $D_i$, such that the out-degree of each 
vertex $v\in C^*$ is at least $\frac{d}{10^7}-5 \frac{d}{10^9}$, which is at least $\Delta$ because $d>10^7\Delta$.
% We find now our $(\Delta-1)$ ISM $\mathcal{I}$, such that for every $x\in C^{j}_i\setminus C_i $, $d^{+}_{\mathcal{I}}(x)=\Delta-1$. \\
We obtain $I'$ from $I$ by deleting all edges whose initial vertex is not in $C^*$.
We let $B_{i+1}$ be the union of $B_i$ and $I'$.

Crucially $I'$ agrees with $D_i = K(B_i)$ because $I$ agrees with $D_i$ and so $B_{i+1}$ is an escape-way by Proposition \ref{prop:ISM_equiv}.
We now let $T_{i+1} \coloneqq T_i\cup xy$ for some $y\in N_{B_{i+1}}(x)$, so that clearly $T_{i+1}\subset B_{i+1}$ as required by \eqref{item:tree_is_sub}.
To see \eqref{item:reserved_enough} we first note it holds by induction for $v\in C_i$ because $B_i\subset B_{i+1}$ and for $v\in C^*$ by the construction of $I'$.
Similarly for \eqref{item:critical_is_out}. 
This completes the induction step.

Thus we obtain our tree $T_{n}$ as a subgraph of our escape-way $B_{n}$ in $G$ and $J$. 
We must check that is induced. 
First note that because $\deg_{B_i}^+(r) = \deg_G(r)$ and all other neighbours were deleted, we have that the in-degree of $r$ is zero. 
By Proposition \ref{prop:} this implies that the component of $B_n$ containing $r$ is a tree.
Further by the definition of escape-way, $\Vin(B_n)$ induces a subgraph of $B_i$, thus all that remains to note is that any neighbours of $r$ that are in the tree are out-neighbours of $r$.
This completes the proof.

\end{proof}

\begin{theorem}\label{thm:countable}
    There exists $\eps>0$ such that if 
    $G$ is a $d$-regular graph with $h(G) > d - 3 + \frac{1}{10^7\Delta+1}$ for some integer $\Delta < \eps d$,
    then $G$ contains a spanning $\Delta$-ary pseudoforest $F$, that is component-wise induced, with the property that one can turn $F$ into an induced forest by deleting one vertex from each of its components.
\end{theorem}
\begin{proof}
    We run the same argument as in the proof of Theorem \ref{thm:main_simp}, with some minor adjustments. 
    Firstly in the online game, the adversary is now allowed to ask us to add neighbours to any vertex in $G$, this does not effect the process it just means that the current $T_i$ may not be connected. 
    We enumerate the vertices of $G$ and assume that the adversary will ask us to extend the vertices in that order. We always extend from a vertex $\Delta$ times in a row. 
    We choose the root $r$ arbitrarily, but we do not do anything to its neighbours and we just proceed with the induction. 
    Case $1$ is identical but in Case $2$ we only require that the critical sets are finite and we show so by applying Lemma \ref{prop:density_of_critical}.
    It is clear that as our adversary will eventually ask us to extend $T_i$ from every vertex and give it out-degree $\Delta$ that we obtain in the limit a spanning escape-way $D$ with all out-degrees $\Delta$. 
    By Proposition \ref{pseudoforest} this is a spanning pseudoforest, and it is straightforward to see that it is $\Delta$-ary.

    We now describe which vertices to delete. 
    By Proposition \ref{prop:} if a component of $D$ has a cycle then it has no vertex of in-degree $0$. 
    For each component we either delete one vertex from the at most one cycle, or we delete the at most one vertex with in-degree $0$ (if there two vertices of indegree $0$ in the same component, then any path between them must have a vertex with indegree at least $2$).
    Thus all remaining vertices are in $\Vin(D)$ and so the resulting graph is induced.
    Further, there are no cycles because we deleted a vertex from each. 
    \end{proof}

 \section{Induced trees in random graphs}

We will show how we can derive easily Theorem~\ref{thm: random_trees}  from Theorem~\ref{thm:main_simp}. We restate it below for convenience of the reader. 
\setcounter{section}{1}
\setcounter{theorem}{3}
\begin{theorem}
    There is $C>0$, such that for all $\Delta\in \mathbb{N}$ and $d> 2^{20\Delta}$, $G(n,d/n)$ contains all trees with maximum degree at most $\Delta$ and order at most $ \frac{Cn}{d\log^2(d)}$ as induced subgraphs with high probability. 
\end{theorem}

\begin{proof}
    Let $G\sim G(n, d/n)$, where $d>2^{20\Delta}$ and $\Delta\geq 1$. It is a simple observation that the following holds w.h.p. 
    \begin{enumerate}
        \item[(i)]For every $S\subset V(G)$, $e(G[S])= \frac{d|S|^2}{2n}\pm dn/10$;
        \item[(ii)] For every $S\subset V(G)$, with $\frac{n}{200d\log(d)} \leq |S|\leq \frac{n}{100d\log(d)}$, $e(G[S])\leq (1+1/5)|S|$;
        \item[(iii)] There are at most $n/4$ vertices with degree greater than $20d$.
   \end{enumerate}
    Let $G$ satisfy the above. Delete from $G$ all vertices of degree greater than $20d$. By (iii), we obtain an induced subgraph $G'\subset G$ on at least $3n/4$ vertices.
    Let $S\subset V(G')$ be a maximal subset of size at most $\frac{n}{200d\log(d)}$ with $e(G'[S])> (1+1/5)|S|$. By assumption $|S|\leq \frac{n}{200d\log(d)}$.  Delete $S$ and let $G''\coloneqq G[V(G')\setminus S]$. It is clear that by maximality and (ii) no $S'\subset V(G'')$ of size at most $\frac{n}{200d\log(d)}$ spans more than $(1+1/5)|S|$ edges. Furthermore, by (ii), $|G''|\geq n/2$ and by (i), we know $e(G'')\geq \frac{d}{8}|G'|$. Finally, passing to an induced subgraph $G'''\subset G''$ with minimum degree at least $d/16$, we have thus constructed an induced subgraph of $G(n,d/n)$ with $\delta(G''')\geq d/16$ and $\Delta(G''')\leq 20d$. Moreover, (by (i)), we know $|G'''|\geq n/50$.  
    It is easy to see $G'''$ satisfies the conditions of Theorem~\ref{thm:main_simp} with $n\coloneqq \frac{n}{10^{14}d\log^2(d)}$ and $\Delta\coloneqq 10^6\log(d)$.
    
\end{proof}

\section{Induced size Ramsey of trees }
In this section, we will prove Theorem~\ref{thm: sizeramsey} which trivially implies Theorem~\ref{thm: 2coloursizeramsey} by taking $\varepsilon= 1/2$. 
As above, we restate the theorem for convenience of the reader. 
\setcounter{section}{1}
\setcounter{theorem}{5}
\begin{theorem}
     For all $\Delta,n\in \mathbb{N}$ and $\eps >0$ there exists a graph $G$ with less than $C(\Delta,\eps) \cdot n$ edges such that any subgraph $J\subset G$ containing $\eps\cdot e(G)$ edges contains every tree of maximum degree $\Delta$ and order at most $n$ as an induced subgraph of $G$. One can take $C(\Delta,\eps) = \left(10^{42}\Delta^3\log(\Delta)\log(\frac{1}{\varepsilon})^3\right)/\eps^2$.
\end{theorem}

\begin{proof}
Let $N\coloneqq \frac{10^{30}n \log(\Delta)\Delta^2 \log(\frac{1}{\varepsilon}))^3}{\varepsilon}$ and $d\coloneqq \frac{10^{12}\Delta\log(\frac{1}{\varepsilon})}{\varepsilon}$
Let $G\sim G\left(N,d/N\right)$ then as above we know that w.h.p. the following holds.
\begin{enumerate}
        \item[(i)]For every $S\subset V(G)$, $e(G[S])= \frac{d|S|^2}{2N}\pm dN/10$;
        \item[(ii)] For every $S\subset V(G)$, with $\frac{N}{200d\log(d)}\leq |S|\leq \frac{N}{100d\log(d)}$, $e(G[S])\leq (1+1/5)|S|$;
        \item[(iii)] There are at most $N/4$ vertices with degree greater than $20d$.
   \end{enumerate}
Let $G$ satisfy the above and delete all vertices of degree greater than $20d$. We obtain an induced subgraph $G'$ on at least $3N/4$ vertices with $\Delta(G_1)\leq 20d$.
Let $S\subset V(G')$ be a maximal subset of size at most $\frac{N}{200\log(d)}$ with $e(G'[S])> (1+1/5)|S|$. By assumption $|S|\leq \frac{n}{200d\log(d)}$. Delete $S$ and let $G''\coloneqq G[V(G')\setminus S]$. It is clear that by maximality and (ii) no $S'\subset V(G'')$ of size at most $\frac{n}{200d\log(d)}$ spans more than $(1+1/5)|S|$ edges. Furthermore, by (ii), $|G''|\geq N/2$ and by (i), we know $e(G'')\geq \frac{d}{8}|G'|$. Finally, passing to an induced subgraph $G'''\subset G''$ with minimum degree at least $d/16$, we have thus constructed an induced subgraph of $G(N,d/N)$ with $\delta(G''')\geq d/16$ and $\Delta(G''')\leq 20d$. Moreover, (by (i)), we know $|G'''|\geq N/50$. $G'''$ will be the desired graph. 

All we need to show is that given any $J\subset G'''$ with $e(J)\geq \varepsilon e(G'')$, $J$ contains an induced copy (in $G'''$) of every tree on $n$ vertices and maximum degree $\Delta$. We first let $J'\subset J$ be an induced subgraph of $J$ with minimum degree $\varepsilon d/20\geq 10^7(10^6\Delta\log(1/\varepsilon)) =10^7f$, where $f\coloneqq 10^6\Delta\log(1/\varepsilon)$. Moreover, by assumption $\Delta(G[V(J')])\leq 20d\leq 2^{f/10^9}$ and
every subset $|S|$ of size at most $10^{10}n\Delta\log(\Delta)\log(\frac{1}{\varepsilon})\geq (10^7f+1)n$ spans at most $(1+1/5)|S|$ edges. We may now invoke Theorem~\ref{thm: main_simplev2} with $G'''\coloneqq G$, $J\coloneqq J$, $\Delta\coloneqq f$ and $n\coloneqq n$. 
\end{proof}

We now easily derive an induced size ramsey result for $q$ colours by taking $\varepsilon\coloneqq 1/q$ in Theorem~\ref{thm: sizeramsey}

\begin{theorem}\label{thm: sizeramseyqcolours}
    There is $C>0$ such that the following holds. Let $q, \Delta, n\geq 1$. Then, there is a graph $G$ on at most $C\Delta^3\log(\Delta)q^2\log(q)^3n$ edges such that in every $q$-edge-colouring of $G$ there is a colour class which spans all trees $T$ on at most $n$ vertices and $\Delta(T)\leq \Delta$ as induced subgraphs of $G$. 
\end{theorem}
We observe this is almost tight as a function of $q$ since even for the non-induced case one has that the $q$ size ramsey number of a path on $n$ vertices is at least $cq^2n$, for some absolute $c>0$.

\section{Concluding remarks}
We have developed an algorithmic approach to embed bounded degree trees in sparse expanding graphs, generalising the remarkable result of Friedman and Pippenger \cite{friedman1987expanding}. 
We have applied this result to give the state of the art on the bounds for multiple questions.
We will now discuss some further avenues of research and state some open problems. 
\subsection{Tightening Theorem~\ref{thm:main_simp}}
There are three main places where one could tighten Theorem \ref{thm:main_simp}, each of which would yield slight but growing (as a function of $\Delta$ or $d$) improvements in our applications. 
Firstly, one could try remove the maximum degree condition.
We use this exclusively when applying the LLL in Lemma \ref{lem: ISM}.
Secondly, one could try to remove the factor of $\Delta$ in the order of the sets upon which we place our density constraint.
Morally, this factor comes from reserving $\Delta$ neighbours for critical vertices even though some critical vertices may never actually be extended from. 
Finally, one could try to replace the average degree upper bound of $12/5$ by some larger constant.
This upper bound is used in three places. 
 It is  used to bound the degeneracy in Lemma \ref{lem: ISM}, but there is a lot of slack in this application. 
 That is, one merely requires that $\Delta$ is much smaller than $d/C^2$, where $d$ is the minimum of the graph given to Lemma \ref{lem: ISM} (we think the factor of $C^2$ is roughly tight here).
 It is also used within Lemma \ref{lem: ISM} when we bound the Lipschitz constant, but this could be avoided by simply requiring girth at least $5$. 
 The place where it is really required is when it used to show that criticality cascades eventually stop (Lemma \ref{prop:density_of_critical}).
 In the case of Theorem \ref{thm:main_simp} there may be an interesting dependency between the average degree upper bound and the order of the sets that must satisfy it. 
 We find this latter problem particularly intriguing.
 For example, in a more concrete way we could not answer the following nice question. 
 \begin{problem}
   Let $d_1,d_2$ be positive integers. Is there $f(d_1,d_2)\geq 1$ and $\varepsilon(d_1,d_2)>0$ such that the following holds. Let $G$ be a graph with average degree $f(d_1,d_2)$ such that all subset of size at most $n$ have average degree at most $d_2$. Is there an induced subgraph $G'\subset G$ of average degree at least $d_1$ such that \textbf{all} subsets $S\subset V(G')$ of size at most $\varepsilon(d_1,d_2)n$ in $G'$ span at most $3/2|S|$ edges? 
 \end{problem}

Finally, we mention a related problem. Theorem \ref{thm:countable} 
in its current form is almost tight. However, if we impose girth at least $5$ then can we apply it to graphs with $h(G) > d - f(d)$ for any function $f$ that tends to infinity with $d$? 
We see no reason why $f$ could not be taken a polynomial.

\subsection{Rolling backwards}
In \cite{draganic2022rolling},  Friedman-Pippenger result was cleverly combined with a ``rolling-back" technique which allowed them to find different structures in expanding graphs. 
It is natural to ask whether our method also allows for roll-backs and
indeed, it does. The key points are that in the proof of the main result we induct on properties of $T_i$, $C_i$ and $B_i$, while $C_i$ can be viewed as a monotone function of $T_i$. 
When rolling back, one \textit{deletes} some vertices of the current embedded tree $T_i$ to obtain $T'$ and only keep the directed edges of $B_i$ that start from the new $C_i$ which is in turn a function of $T'$.
The desired properties are maintained. 
One can even delete non-leaf vertices when rolling back although the bound on the order of the critical bootstrap percolation will be worse (larger) than the factor of $2$ in Lemma \ref{prop:critical_bounds_available} (which uses that $G[X]$ is connected).

\subsection{Induced ramsey and induced size ramsey}
Recently, Dragani\'c and Keevash~\cite{keevash2024size} gave a bound on the induced size ramsey number of paths, $\rsind^q(P_n)=O(nq^3\log^2(q))$. Theorem~\ref{thm: sizeramseyqcolours} gives an improvement on this by essentially a factor of $q$. 
We believe however the size ramsey and induced size ramsey of bounded degree trees should not behave very differently \textit{as a function of the number of colours.} 
\begin{conjecture}
    Let $\Delta\geq 1$ and $T$ be a tree on $n$ vertices with $\Delta(T)\leq \Delta$. Then, for every $q\geq 1$,$$\rsind^q(T)=O_\Delta(\rsi^q(T)).$$ 
\end{conjecture}

We reiterate a central problem in the area regarding the induced ramsey number of  bounded degree graphs. 
The best upper bound on the induced ramsey number  for graphs of bounded degree is $n^{O(\Delta)}$ proved by Conlon, Fox and Zhao~\cite{conlon2014extremal}. 
It is therefore remarkable that the possibility of induced ramsey numbers of bounded degree graphs being linear remains open.
\begin{problem}
    Is there $\Delta\geq 1$, such that for every $C>0$, there is a graph $G$ on $n$ vertices and $\Delta(G)\leq \Delta$ with $\rind(G)>Cn$?
\end{problem}

\subsection{Induced structures in random graphs}

We are confident our main result will be very useful in finding other large induced structures in random graphs. 
For example, one could ask what is the largest $k$ for which $G(n,p)$ contains an \textit{induced} subdivision of a $K_k$ whp? 
This would be a induced version of a classical result of Ajtai, K\'omlos and Szemer\'edi~\cite{AKS1979topological} which guarantees that whp $G(n,p)$ where $p=o(\frac{1}{{\sqrt{n}}})$ contains a subdivison of $K_{(1+o(1))\Delta}$ where $\Delta $ is the maximum degree of $G(n,p)$. 
Finally, we think our methods could be helpful in proving essentially tight bounds for the size of induced bounded degree trees in $G(n,p)$. 

\begin{conjecture}
    For every $\Delta\geq 1$ there is $C_{\Delta}>0$ such that the following holds. For all $\frac{C_{\Delta}}{n}\leq p\leq 0.1$, $G(n,p)$ contains w.h.p. all trees of order $\Omega\left (\frac{\log(pn)}{p}\right )$ with maximum degree $\Delta$. 
\end{conjecture}
We do not even know the above result for any bounded degree tree (including a path).

\bibliographystyle{alpha}
\bibliography{Induced}

\end{document}